\documentclass[11pt,reqno]{amsart}
\usepackage{amsmath}
\usepackage{amssymb}
\usepackage{amsthm}
\usepackage{enumerate}
\usepackage[mathscr]{eucal}

\usepackage{mathtools}
\usepackage{microtype}

\usepackage{siunitx}

\usepackage{times}
\usepackage[english]{babel}

\usepackage{amsmath,amssymb,latexsym,textcomp,mathrsfs}
\usepackage[all]{xy}
\usepackage{graphicx}
\usepackage{bm,amsmath, amsthm, amssymb, amsfonts}
\usepackage{times}
\usepackage[english]{babel}

\setbox0=\hbox{$+$}
\newdimen\plusheight
\plusheight=\ht0
\def\+{\;\lower\plusheight\hbox{$+$}\;}  

\setbox0=\hbox{$-$}
\newdimen\minusheight
\minusheight=\ht0
\def\-{\;\lower\minusheight\hbox{$-$}\;}

\setbox0=\hbox{$\cdots$}
\newdimen\cdotsheight
\cdotsheight=\plusheight
\def\cds{\lower\cdotsheight\hbox{$\cdots$}}

\newcommand{\norm}[1]{\left\lVert#1\right\rVert}

\numberwithin{equation}{section}
\theoremstyle{plain}
\newtheorem{thm}{Theorem}[section]
\newtheorem{lem}[thm]{Lemma}

\newtheorem{cor}{Corollary}

\theoremstyle{definition}
\newtheorem{defn}{Definition}[section]

\newtheorem{exmp}{Example}[section]

\theoremstyle{remark}
\newtheorem{rem}{\bf{Remark}}
\newtheorem*{note}{\bf{Note}}
\numberwithin{equation}{section}

\begin{document}
\setcounter {page}{1}
\title{\bf{$I$}-divergence and \bf{$I^*$}-divergence in cone metric spaces}

\author{Amar Kumar Banerjee and Anirban Paul}

\address[A.K.Banerjee]{Department of Mathematics, The University of Burdwan, Golapbag, Burdwan-713104, West Bengal, India.}

\email{ akbanerjee1971@gmail.com, akbanerjee@math.buruniv.ac.in }

\address[A.Paul]{Department of Mathematics, The University of Burdwan, Golapbag, Burdwan-713104, West Bengal, India.}
         
\email{ paulanirban310@gmail.com, anirbanpaul310math@gmail.com}

\begin{abstract} In this paper we have studied the ideas of $I$-divergence and $I^*$-divergence of sequences in cone metric spaces. We have investigated the relationship between $I$\text-divergence and $I^*$-divergence and their equivalence under certain condition. Further we prove a decomposition theorem for $I$-convergent sequences in a cone normed space.
\end{abstract}
\maketitle
\author{}


\textbf {(2010) AMS subject classification :} 54A20, 40A35  \\
 
\section{Introduction}
The idea of statistical convergence of sequences of real number were introduced by Steinhaus \cite{q} and Fast \cite{f} as a generalizations of ordinary convergence of sequences of real numbers. Later many more works were done in this direction \cite{B, g, p}. In 2001, P. Kostryko et al \cite{h} introduced the idea of $I$-convergence of real sequences using the ideals of the set of natural numbers as a generalization of statistical convergence. Later in 2005, Lahiri and Das \cite{j} studied the same in a topological space and then many works were carried out in this direction \cite{2, 4, d, i}. In 2017, Banerjee and Mondal \cite{3} studied the same for double sequences in a topological space.\\ 

The concepts of divergent sequences of real numbers was generalized to statistically divergent sequences of real numbers by Macaj and Salat in \cite{l}. Recently Das and Ghosal in \cite{c} introduced the notion of $I$ and $I^*$-divergence of sequences in a metric space where the condition (AP) has been used as necessary and sufficient condition to prove the equivalence of $I$ and $I^*$-divergence.\\

In 2007, Nabiev et al \cite{m} established a decomposition theorem for $I$-convergent sequences in a linear metric space. The concept of cone metric spaces is a generalization of the notion of usual metric spaces where the distance between two points is given by an element of a Banach space endowed with suitable partial ordering. After the initial work on cone metric spaces space by Guang and Xian \cite{k}, a lot of work have been done on this structure. In this paper we proceed in a different direction and extend the notion of $I$ and $I^*$-divergence in a cone metric space.\\

The idea of cone normed spaces was given by M. Eshaghi Gordrji et al \cite{g*} which is a generalization of normed spaces. In section \ref{sec} we have proved the decomposition theorem of $I$-convergence  sequences in a cone normed space and have studied its consequences, where the methods of proof are not analogous as in the case of a metric space.
\section{Preliminaries}
The following definitions and notion will be needed.\\

\begin{defn} \cite{t} Let $X\neq\phi$. A family $I \subset 2^X$ of subsets of $X$ is said to be an ideal in $X$ provided the following conditions holds:\\
\\
$(i)\: A, B \in I \Rightarrow A\cup B\in I$\\
$(ii) \: \:A\in I, B\subseteq A \Rightarrow B\in I$
\end{defn}
Note that $\phi\in I$ follows from the condition(ii). An ideal $I$ is called nontrivial if $I \neq \left\{\phi\right\}$ and $X\not\in I$. $I$ is said to be admissible if $\left\{x\right\}\in I$ for each $x\in X$ \cite{t**}.
\begin{defn} \cite{t*} Let $X\neq\phi$. A non-empty family $F\subset 2^X$ is said to be a filter on $X$ if the following are satisfied,\\
$(i)\: \phi\not\in F $\\
$(ii)\: A, B \in F \:\Rightarrow \:A\cap B\in F $\\
$(iii)A\in F, A\subset B\: \Rightarrow \: B\in F$
\end{defn}
Clearly if $I$ is a nontrivial ideal then the family of sets $F(I)=\left\{M\subset X : \text{there exists}\: A\in I, M=X\setminus A\right\} $ is a filter in $X$. It is called the filter associated with the ideal $I$.

\begin{defn} \cite{h,t**} Let $I\subset 2^\mathbb{N}$ be a proper ideal in $\mathbb{N}$ and $(X,d)$ be a metric space. The sequence $\left\{x_n\right\}_{n\in \mathbb{N}}$ of elements of $X$ is said to be $I$-convergent to $x\in X$ if for each $\varepsilon>0$ the set $A(\varepsilon)=\left\{n\in \mathbb{N}:d(x_n,x)\geq \varepsilon\right\}$ belongs to $I$.
\end{defn}

\begin{defn} \cite{h}\label{1} An admissible ideal $I \subset 2^\mathbb{N}$ is said to satisfy the condition (AP) if for every countable family of mutual disjoint sets $\left\{A_1,A_2,A_3,\cdots\right\}$ belonging to $I$ there exists a countable family of sets $\left\{B_1,B_2,B_3,\cdots\right\}$ such that $A_j\Delta B_j$ is a finite set for each $j\in\mathbb{N}$ and $B=\displaystyle{\bigcup_{j=1}^{\infty}}B_j\in I$.
\end{defn}
Note that $B_j\in I$ for each $j\in \mathbb{N}$.\\

The concepts of $I^*$-convergence which is closely related to the $I$-convergence has been given in \cite{h} as follows:\\
\begin{defn} \cite{h} The sequence $\left\{x_n\right\}_{n\in\mathbb{N}}$ of elements of X is said to be $\text{I}^*$\text-convergent to $x\in X$ if and only if there exists a set $M\in F(I)$, $ M=\left\{m_1<m_2 \cdots <m_k\cdots \right\}\subseteq \mathbb{N}$ such that $\displaystyle{\lim_{k\to\infty}}d(x,x_{m_k})=0$.
\end{defn}

In \cite{h} it is seen that $I$ and $I^*$-convergence are equivalent for an admissible ideal with the property (AP).\\

Now we recall the following definitions.
\begin{defn} \cite{1} A sequence $\left\{x_n\right\}_{n\in \mathbb{N}}$ in a metric space $(X,d)$ is said to be divergent (or properly divergent) if there exists an element $x \in X$ such that $d(x,x_n)\to \infty$ as $n\to\infty$.\\

\end{defn}
Note that a divergent sequence in a metric spaces can not have any convergent subsequence.
\begin{defn} \cite{c} A sequence $\left\{x_n\right\}_{n\in \mathbb{N}}$ in  a metric space $(X,d)$ is said to be $I$\text-divergent if there exists an element $x\in X$ such that for any positive real number G, $A(x,G)=\left\{n \in \mathbb{N}:d(x,x_n)\leq G\right\}\in I$
\end{defn}
\begin{defn} \cite{c} A sequence in a metric space $(X,d)$ is said to be $I^*$-divergent if there exists $M\in F(I)$ i,e $\mathbb{N}\setminus M\in I$ such that $\left\{x_n\right\}_{n\in M}$ is divergent. That is there exists at least one $x\in X$ such that $\displaystyle{\lim_{n\to\infty}}d(x,x_n)=\infty$
\end{defn}
We now state the decomposition theorem from \cite{m} for $I$-convergent sequences.\\

\begin{thm} \cite{m}
Let $(X,d)$ be a linear metric space, $x=(x_n)\in X$ and $I\subset 2^\mathbb{N}$ be an admissible ideal with the property (AP). Then the following are equivalent:\\
(a)\: $\text{I}\text -lim x_n=\xi$\\
(b)\: There exists $y=(y_n)\in X$ and $z=(z_n)\in X$ such that $x=y + z $, $\displaystyle{\lim_{n\to\infty}}d(y_n,\xi)=0$ and $\text{supp z}\in I$, where $\text{supp z}=\left\{n\in \mathbb{N}:z_n\neq \theta_X\right\}$ and $\theta_X$ is the zero element of $X$.

\end{thm}

We now recall the following basic concepts from \cite{k} which will be needed throughout the paper. $E$ will always denote a real Banach space with null vector $\theta_E$ and P be a subsets of $E$. P is called cone if and only if \\
(i) P is closed and $P\neq \left\{\theta_E\right\}$\\
(ii) $a, b\in \mathbb{R} \text{,} \: a, b\geq 0 \text{,} \:x,y\in P \Rightarrow ax+by \in \text{P}$\\
(iii) $x\in \text{P}$ and $-x\in \text{P}$ $\Rightarrow$ $x=\theta_E$\\

Given  cone $\text{P}\subset E$, we define a partial ordering $\leq$ with respect to P by $x\leq y$ if and only if $y -x \in \text{P}$. We shall write $x < y$ to indicate that $x\leq y$ but $x\neq y$, while $x<< y$ will stand for $y - x \in \text{Int P}$, Int P denotes the interior of P.\\

The cone P is called normal if there is a number $k>0$ such that for all $x, y \in E$\\
$\theta_E\leq x \leq y$ implies $\norm x \leq k\norm y$.\\
The least positive numbers satisfying above is called the normal constant of P.\\
The cone is called regular if every increasing sequence which is bounded from above is convergent. That is if $\left\{x_n\right\}$ is a sequence such that $$x_1\leq x_2\leq \cdots \leq x_n\leq \cdots \leq y$$ for some $y \in E$, there is $x\in E$ such that $\norm {x_n - x }\to 0 \: (n \to \infty)$. Equivalently the cone P is regular if and only if every decreasing sequence which is bounded from below is convergent. It is well known that a regular cone is a normal cone.\\

In the following we always suppose $E$ is a real Banach space, with null element $\theta_E$, P is a cone with $\text{Int P} \neq \phi$ and $\leq$ is a partial ordering with respect to P.\\
\begin{defn} \cite{k} Let $X$ be a non-empty set. Suppose the mapping $d : X\times X \mapsto E $ satisfies\\
(d1) $\theta_E< d(x,y)$ for all $x, y \in X$ and $d(x,y)=\theta_E$ \:if and only if\: $x=y$\\
(d2) $d(x , y)=d(y , x)$ for all $ x, y \in X$\\
(d3) $d(x , y)\leq d(x , z) + d(y , z) $ for all $x, y, z \in X $.\\
Then $d$ is called a cone metric and $(X,d)$ is called a cone metric space. It is obvious that cone metric spaces generalize metric spaces. Several examples of cone metric space are seen in \cite{k}.
\end{defn}

\begin{defn} \cite{k} Let $(X,d)$ be a cone metric space. Let $\left\{x_n\right\}$ be a sequence in $X$ and $x\in X$. If for every $c\in E$ with $\theta_E<<c$ there is $N$ such that for every $n>N$, $d(x,x_n)<< c$, then $\left\{x_n\right\}$ is said to be convergent and $\left\{x_n\right\}$  converges to $x$ and $x$ is the limit of $\left\{x_n\right\}$ . We denote this by $\displaystyle{\lim_{n\to\infty}}x_n=x$ or $x_n\to x \: (n\to\infty)$
\end{defn}
\begin{lem} \cite{k} Let $(X,d)$ be a cone metric space. P be a normal cone with normal constant $k$. Let $\left\{x_n\right\}$ be a sequence in $X$. Then $\left\{x_n\right\}$ converges to $x$ if and only if $d(x,x_n)\to \theta_E \: (n\to\infty)$.
\end{lem}
We now consider the following definitions. \\

\begin{defn} \cite{s} Let $X$ be a real vector space . Suppose that the mapping $\norm . _\text{P} : X\mapsto E $ be such that 
\begin{align*}
&(i)\norm {.}_\text{P} > \theta_E\: \text{for all} \: x\in X \text{and} \norm . _\text{P}= \theta_E \:\text{if and only if} \: x=\theta_X\\
&(ii) \norm {\alpha x}_\text{P} = |\alpha|\norm{x}_{\text{P}} \text{for all} \:x\in X\text{ and}\: \alpha \in \mathbb{R}\\
&(iii) \norm {x + y}_{\text{P}} \leq \norm{x}_{\text{P}} + \norm{y}_{\text{P}}, \text{for all }x, y\in X.
\end{align*}

Then $\norm{.}_{\text{P}}$ is called a cone norm on $X$ and $(\norm{.}_{\text{P}}, X)$ is called cone normed space.\\
\end{defn}
It is easy to show that every normed space is cone normed space by putting $E=\mathbb{R}$, $\text{P}=[0,\infty)$.\\
\begin{rem} \cite{s} Let $(X,\norm {.}_{\text{P}})$ be a cone normed space, set $d(x , y) =\norm {x - y}_{\text{P}}$, it is easy to show that $(X,d)$ is a cone metric space, $d$ is called `` the cone metric induced by the cone norm $\norm {.}_{\text{P}}$".
\end{rem}
Now we consider the following theorem as in \cite{s}.
\begin{thm} \cite{s} The cone metric $d$ induced by a cone norm on a cone normed space satisfies \\
(i) $d(x+a,y+a) = d(x,y), \: x, y, a\in X$\\
(ii) $d(\alpha x, \alpha y )=|\alpha | d(x , y), \: x, y\in X,\: \alpha \in \mathbb{R}$

\end{thm}

Now we consider the following definition from \cite{n}.\\
\begin{defn} \cite{n} Let $(X, d)$ be a cone metric space. Let $\left\{x_n\right\}_{n\in \mathbb{N}}$ be a sequence in $X$ and let $x \in X$. If for every $c\in E$ with $\theta_E << c$, the set $\left\{ n\in \mathbb{N} : c - d(x,x_n)\not\in \text{Int P}\right\} \in I$ then $\left\{x_n\right\}_{n\in \mathbb{N}}$ is said to be $I$-convergent to $x$ and we write $I\text-\displaystyle{\lim _n} x_n=x$
\end{defn}
\begin{defn} \cite{n} A sequence  $\left\{x_n\right\}_{n\in \mathbb{N}}$ in $X$ is said to be $I^*$-convergent to $x\in X$ if and only if there exists a set $M \in F(I), M=\left\{m_1<m_2<\cdots <m_k<\cdots\right\}$ such that $\displaystyle{\lim_{k\to \infty}}x_{m_k}= x$ i,e for every $c\in E$ with $\theta_E<<c$, there exists $p\in \mathbb{N}$ such that $ c - d(x,x_{m_k})\in \text{Int P}$ for all $k\geq p$. 
\end{defn}

It is known \cite{r} that any cone metric space is first countable Hausdorff topological space with the topology induced by the open balls defined as usual for each element $z$ in $X$ and for each element $c$ in Int P. So as in \cite{j} it can be seen that $I^*$-convergence always implies $I$-convergence but the converse is not true. The two concepts are equivalent if and only if the ideal $I$ has the condition (AP).\\

Now we consider the following theorems as in  \cite{5}.
\begin{thm} \cite{5}
Let $E$ be a real Banach space with cone P. If $x_0\in \text{Int P}$ and $c(>0)\in \mathbb{R}$ then $c x_0\in \text{Int P}$.
\end{thm}
\begin{thm} \cite{5}
Let $E$ be a real Banach space and P be a cone in $E$, if $x_0\in \text{P}$ and $y_0\in \text{Int P}$ then $ x_0 + y_0 \in \text{Int P}$
\end{thm}

\section{$I$-divergence and $I^*$-divergence}

Throughout the paper $X$ stands for a cone metric space with normal cone P and the cone metric $d : X\times X \mapsto E$, $\theta_E$ stands for the null vector of $E$, $I$ for nontrivial admissible ideal of $\mathbb{N}$, the set of natural numbers unless otherwise stated.

\begin{defn} A sequence $\left\{x_n\right\}_{n\in \mathbb{N}}$ in a cone metric space $(X,d)$ is said to be divergent (or properly divergent) if there exists an element $x\in X$ such that for any $c \in E$ with $\theta_E<< c$ there exists $n_0 \in \mathbb{N}$ such that $d(x,x_n) - c\in \:\text{Int P}$ for all $n\geq n_0$.
\end{defn}
\begin{exmp}Let $E=\mathbb{R}^{2}$ with usual norm, $\text{P}=\left\{(x,y)\in E : x,y \geq 0\right\}\subset \mathbb{R}^{2}$. $X=\mathbb{R}^{2}$ and $d: X \times X \mapsto E$ such that $d(x,y )=( |x_1 - y_1 | , |x_2 - y_2| )$, where $x=(x_1 ,x_2)$, $y=(y_1,y_2)$. Let us define a sequence $\left\{x_n\right\}_{n\in \mathbb{N}}$ as follows $x_n=(n,n)$. Then there exists $x=(0,0)$ such that for any $c\in E $ with $\theta_E<<c$ there exists $n_0=\text{max}\left\{c_1, c_2\right\} + 1 \in \mathbb{N}$, where $c=(c_1, c_2)$ such that $d(x , x_n) - c =(n,n) -c \in \:\text{Int P}$ for all $n\geq n_0$. 
\end{exmp}
Now we have the following theorem related to divergent sequence in a cone metric space.\\
\begin{thm} A divergent sequence can not have a convergent subsequence in a cone metric space.
\end{thm}
\begin{proof}
Let $\left\{x_n\right\}_{n\in \mathbb{N}}$ be a sequence in a cone metric space which is divergent. Let $c\in E$ with $\theta_E<< c$. Since  $\left\{x_n\right\}_{n\in \mathbb{N}}$ is divergent there exists an element $x\in X$ such that for this $c\in E$ there exists $n_0\in \mathbb{N}$ such that $d(x,x_n) - c \in \:\text{Int P}$ when $n\geq n_0$. Now if possible let $\left\{x_{n_k}\right\}_{k\in \mathbb{N}}$ be a subsequence of  $\left\{x_n\right\}_{n\in \mathbb{N}}$ converging to some $x \in X$. Then for this $c\in E$ we have a $p \in \mathbb{N}$ such that $d(x,x_{n_k})<< c$ for all $n_k \geq n_p$. Let r = max $\left\{n_0, n_p\right\}$. Then $d(x,x_r) - c \in \: \text{Int P}$ and $ c - d(x,x_r)\in \: \text{Int P}$ i,e $-(d(x,x_r) - c)\in \: \text{Int P}$. So $\theta_E \in \: \text{Int P}$, which is a contradiction. Hence the result follows.
\end{proof}
\begin{defn} A sequence  $\left\{x_n\right\}_{n\in \mathbb{N}}$ in a cone metric space $(X,d)$ is said to be $I$-divergent if there exists an element $x\in X$ such that for any $c\in E$ with $\theta_E<< c$, $A(x,c)=\left\{ n\in \mathbb{N} : c -d(x,x_n)\in \:\text{Int P}\right\} \\
\in I$.
\end{defn}
\begin{defn} A sequence  $\left\{x_n\right\}_{n\in \mathbb{N}}$ in a cone metric space $(X,d)$ is said to be $I^*$-divergent if there exists an $M\in F(I)$ (i,e $\mathbb{N}\setminus M \in I$ ) such that  $\left\{x_n\right\}_{n\in \mathbb{N}}$ is divergent i,e there exists an element $x\in X$ such that for any $c\in E$ with $\theta_E<< c$ there exists $n_0\in \mathbb{N}$ such that $d(x,x_n) - c\in \:\text{Int P}$ $\forall n \geq n_0, n\in M$.
\end{defn}
\begin{thm}\label{th2} Let $I$ be an admissible ideal. If $\left\{x_n\right\}_{n\in \mathbb{N}}$ is $I^*$ -divergent then  $\left\{x_n\right\}_{n\in \mathbb{N}}$ is $I$-divergent.
\end{thm}
\begin{proof}
Since  $\left\{x_n\right\}_{n\in \mathbb{N}}$ is $I^*$-divergent, there exists an $M\in F(I) \:  (i,e \mathbb{N}\setminus M\in I)$ such that  $\left\{x_n\right\}_{n\in M}$ is divergent. That is there exists an $x\in X$ such that for any $c\in E$ with $\theta_E<< c$ $\exists$ \: $n_0\in \mathbb{N}$ such that $d(x,x_n) - c\in \:\text{Int P}$ for all $n\geq n_0 $,$n\in M$. Now the set $A(x,c)=\left\{ n\in \mathbb{N}: c - d(x,x_n)\in \: \text{Int P}\right\}\subset (\mathbb{N}\setminus M)  \cup \left\{1,2,\cdots, n_0\right\}$. Since $I$ is an admissible ideal the set $\mathbb{N}\setminus M \cup\left\{1,2,\cdots,n_0\right\}\in I$. So $A(x,c)\in I$. This implies that $\left\{x_n\right\}_{n\in \mathbb{N}}$ is $I$-divergent.
\end{proof}
The converse of the above theorem is not always true. The following is one such example.\\

\begin{exmp} 
Let $\mathbb{N}=\displaystyle{\bigcup_{j\in \mathbb{N}}}\Delta_j$ be a decomposition of $\mathbb{N}$ such that $\Delta_j$ is infinite and $\Delta_i \cap\Delta_j=\phi$ if $i\neq j$. Let $I$ be the class of all those subsets of $\mathbb{N}$ that intersects with only finite numbers of $\Delta_i$'s. Then $I$ is a non-trivial admissible ideal of $\mathbb{N}$. Now let us consider $X=\mathbb{R}$, $E=\mathbb{R}^2$, with usual norm and $d(x,y )=( |x - y| , |x - y|)$, $P=\left\{ (x , y) : x,y \geq 0\right\}$. Then $(X,d)$ is a cone metric space. Now let us construct a sequence as follows: $y_i = n$ if $i \in \Delta_n$. Now for any $c\in E$ with $\theta_E<< c$ the set $\left\{ i \in \mathbb{N} : c - d(0,y_i)\in \: \text{Int P}\right\}$ is contained in the union of finite numbers of $\Delta_j$'s and hence belongs to $I$. So $\left\{y_i\right\}_{i\in \mathbb{N}}$ is $I$-divergent. Next we shall show that $\left\{y_i\right\}_{i\in \mathbb{N}}$ is not $I^*$-divergent. If possible assume that $\left\{y_i\right\}_{i\in \mathbb{N}}$ is $I^*$-divergent. Then by definition there is a $M\in F(I)$ such that  $\left\{y_i\right\}_{i\in M}$ is divergent. Since $\mathbb{N}\setminus M\in I$ so there exists a $l \in \mathbb{N}$ such that $\mathbb{N}\setminus M \subset \Delta_1 \cup\Delta_2\cup\cdots\cup\Delta_l$. [For, if no such $l$ is found then $\mathbb{N}\setminus M$ intersects  infinitely many $\Delta_j$ which contradicts the fact that $\mathbb{N}\setminus M \in I $]. But then $\Delta_i \subset M \: \forall i > l$. In particular $\Delta_{l+1}\subset M$. But this implies that  $\left\{y_i\right\}_{i\in \Delta_{l+1}}$ is a constant subsequence of  $\left\{y_i\right\}_{i\in M}$ which converges to $l+1$. This contradicts the fact that $\left\{y_i\right\}_{i\in M}$ is divergent.
\end{exmp}
 

\begin{thm}\label{th1} If $I$ is an admissible ideal with property (AP) then any $I$-divergent sequence $\left\{x_n\right\}_{n\in \mathbb{N}}$ in $X$ is $I^*$-divergent if the following additional condition (C) holds for the cone P of a real Banach space $E$,\\
 (C) : For any $c\in E$ and $\eta\in \text{Int P}$ there exists a natural number $K$ such that $c<< K\eta$.
\end{thm}
\begin{proof}
First suppose that $I$ satisfies the condition (AP). Since  $\left\{x_n\right\}_{n\in \mathbb{N}}$ is $I$-divergent, so there exists an element $x\in X$ such that for any $c\in E$ with $\theta_E <<c$,  $A(x,c)=\left\{n\in \mathbb{N} : c - d(x,x_n)\in \: \text{Int P}\right\}\in I$ i,e for any $c\in \text{Int P}$, $A(x,c)=\left\{n\in \mathbb{N} : d(x,x_n) << c\right\}\in I$. Now let us choose a $\eta\in E$ with $\theta_E<< \eta$. For all $k\geq 2$, let\begin{align*}&A_1 = \left\{n\in \mathbb{N} : d(x,x_n) << \eta\right\}\\
&A_2 = \left\{n\in \mathbb{N} : \eta \leq d(x,x_n) << 2\eta\right\}\\
&A_3 = \left\{ n\in \mathbb{N} : 2\eta \leq d(x,x_n)<< 3\eta\right\}\\
&\vdots\\
&A_k = \left\{ n\in \mathbb{N} : (k - 1)\eta \leq d(x,x_n)<< k\eta\right\}.
\end{align*} \\
Since for any $k\geq 2$, $k\eta\in \text{Int P}$ as $\eta\in \text{Int P}$, it follows that $\left\{n\in\mathbb{N} : d(x,x_n) << k\eta \right\}
\in I$. Since $A_k \subset \left\{n\in\mathbb{N} : d(x,x_n) << k\eta \right\}$, $A_k\in I$ \: $\forall \: k\geq 2$; also $A_1\in I$. Now suppose $A_i \cap A_j \neq \phi$. Without loss of generality suppose $i< j$, then $i+1\leq j$. If $j= i + 1$ then $A_i=\left\{ n\in \mathbb{N}: (i - 1) \eta \leq d(x,x_n) << i\eta\right\}$ and $A_j = A_{i+1} = \left\{ n \in \mathbb{N}: i\eta \leq d(x,x_n)<< (i+1)\eta\right\}$. Let $m\in A_i\cap A_j$ then $d(x,x_m)<< i\eta$ and $i\eta\leq d(x,x_m)$ $\Rightarrow$ $i\eta - d(x,x_m) \in \text{ Int P}$ and $d(x,x_m) - i\eta \in \text{P}$ $\Rightarrow$ $\left\{( i\eta - d(x,x_m)) + (d(x,x_m) - i\eta)  \right\}\in \text{Int P}$ $\Rightarrow$ $\theta_E \in \text{Int P}$, a contradiction. Again if $i + 1 < j$ then $(j - i- 1)> 0$, so $\eta(j - i -1)\in \text{Int P}$ i,e $(j - 1)\eta - i\eta \in \text{Int P}$ and hence $i\eta << (j-1)\eta$. Let $m\in A_i \cap A_j$. Now $m\in A_i$ $\Rightarrow$ $d(x, x_m) << i\eta$ which implies that $d(x,x_m) << i\eta << (j-1)\eta $ $\Rightarrow$ $ d(x,x_m) << (j -1)\eta$. Again $m\in A_j$ $\Rightarrow$ $(j - 1)\eta \leq d(x,x_m)$  $\Rightarrow$ $d(x,x_m) - (j-1)\eta \in \text{P}$. Hence $\left\{ ((j -1 )\eta - d(x,x_m)) + ( d(x,x_m) - (j -1 )\eta)\right\}=\theta_E\in \text{Int P}$, a contradiction. Thus we get a collection of mutual disjoint sets $\left\{A_i\right\}_{i\in \mathbb{N}}$ with $A_i\in I$ for all $i\in \mathbb{N}$. By the condition (AP) there exists a family of sets $\left\{B_i\right\}_{i\in \mathbb{N}}$ such that $A_i\Delta B_i$ is finite for all $i$'s and $B=\displaystyle{\bigcup_{i\in\mathbb{N}}}B_i\in I$. Let $M=\mathbb{N}\setminus B $ then $M \in F(I)$. Let us take any $c\in E$ with $\theta_E<< c$. Since the cone P satisfied the condition (C), choose $k\in \mathbb{N}$ such that $c<< k\eta$. Then $ \left\{n\in \mathbb{N}: d(x,x_n) << c\right\} \subset A_1\cup A_2\cup \cdots \cup A_k$. Since $A_i \Delta B_i$ is finite, so there exists $n_0\in \mathbb{N}$ such that $ (\displaystyle{\bigcup_{i=1}^{k}} B_i) \cap \left\{ n\in \mathbb{N} : n\geq n_0\right\} =(\displaystyle{\bigcup_{i=1}^{k}}A_i) \cap\left\{n\in\mathbb{N} : n\geq n_0\right\}$. Clearly if $n\geq n_0$ and $n\in M$ then $ n\not\in \displaystyle{\bigcup_{i=1}^{k}}B_i \Rightarrow n\not\in  \displaystyle{\bigcup _{i=1}^{k}}A_i$. Therefore $d(x,x_n)\geq k\eta>>c$. This implies $d(x,x_n)>>c$. Thus  $\left\{x_n\right\}_{n\in M}$ is $I^*$-divergent. 
\end{proof}
\begin{note}
Note that  the condition (C) stated in the theorem \ref{th1} is precisely the Archimedean condition for $E$ with cone P, also note that there are Banach space and corresponding cones where this additional condition (C) holds as shown in the following example.
\end{note}
\begin{exmp}
Let $E=\mathbb{R}^2$, equipped with usual norm. Let $\text{P}=\left\{ (x,y) : x,y \geq 0\right\}$. Let $y=(y_1, y_2)$ be any element in $E$ and $x=(x_1, x_2)\in \text{Int P}$. Let $r_1=\max\left\{ |y_1|, |y_2|\right\} $ and $r_2=\min\left\{x_1, x_2\right\}$. Therefore $r_2 > 0$ and $x_1 \geq r_2$ and $x_2\geq r_2$. Also $|y_1| \leq r_1$, $|y_2| \leq r_1$. So $y_1 \leq r_1$ and $y_2 \leq r_1$. Choose an natural number $n$ by the Archimedean property of $\mathbb{R}$, such that $n r_2 > r_1$. So $n x_1\geq n r_2 > r_1 \geq y_1$ and $n x_2 \geq n r_2 >r_1 \geq y_2$. So $n (x_1, x_2) - (y_1, y_2) = (n x_1 - y_1 , n x_2 - y_2) \in \text{Int P}$, since $n x_1 - y_1 > 0$ and $n x_2 - y_2> 0$. Hence $n (x_1, x_2) >> (y_1, y_2)$.
\end{exmp}
\begin{thm}\label{th3} Let $(X,d)$ be a cone metric space containing at least one divergent sequence. If every $I$-divergent sequence  $\left\{y_n\right\}_{n\in \mathbb{N}}$ is $I^*$-divergent then $I$ satisfies the condition (AP).
\end{thm}
\begin{proof}
Let $\left\{x_n\right\}_{n\in \mathbb{N}}$ be a divergent sequence in $X$. Then there exists an element $x\in X$ such that for any $c\in E$ with $\theta_E<<c$ there exists $k\in \mathbb{N}$ such that $d(x,x_n) - c\: \in \text{Int P}$ for all $n\geq k$. Suppose $\left\{A_i : i=1,2,3,\cdots\right\}$ is a sequence of mutually disjoint non-empty sets from $I$. Define a sequence $\left\{y_n\right\}_{n\in \mathbb{N}}$ as follows: $y_n = x_j $ if $n\in A_j$ and $y_n = x_n$ if $n\not\in A_j$ for any $j\in \mathbb{N}$. Let $c\in E$ be any element such that $\theta_E<< c$. Now $A(x,c)=\left\{n\in \mathbb{N} : d(x,y_n)<< c\right\} \subset A_1\cup A_2\cup A_3\cup \cdots \cup A_k \cup\left\{ 1,2,\cdots,k\right\}\in I$. So $\left\{y_n\right\}_{n\in \mathbb{N}}$ is $I$-divergent. By our assumption $\left\{y_n\right\}_{n\in \mathbb{N}}$ is $I^*$-divergent. So there exists $M\subset \mathbb{N}$ such that $M\in F(I)$ and $\left\{y_n\right\}_{n\in M}$ is divergent. Let $B=\mathbb{N}\setminus M$. Then $B\in I$. Put $B_j=A_j \cap B$ for all $j\in \mathbb{N}$. Since $\displaystyle{\bigcup_{j\in \mathbb{N}}}B_j \subset B$; $\displaystyle{\bigcup_{j\in \mathbb{N}}}B_j \in I$. Then $A_j\cap M$ is a finite set. For, if possible let $A_j \cap M$ is not finite set, then $M$ must contain an infinite sequence of elements $\left\{m_k\right\}$. So $y_{m_k}=x_j$ for all $k\in \mathbb{N}$, which forms a convergent sub-sequence of  $\left\{y_n\right\}_{n\in M}$. But this contradicts the fact that  $\left\{y_n\right\}_{n\in M}$ is divergent. Hence $A_i\Delta B_i=A_i\setminus B_i =A_i \cap B_i^{c}=  A_i \cap (M\cup A_i^{c})=A_i \cap M$, which is a finite subset of $\mathbb{N}$. This proves that $I$ satisfies the condition (AP).
\end{proof}
\begin{cor} 
If the condition (C) stated in the theorem \ref{th1} holds for a real Banach space and its corresponding cone. Then for an admissible ideal $I$ of $\mathbb{N}$ a $I$-divergent sequence is $I^*$-divergent if and only if $I$ satisfies the condition (AP).    
\end{cor}
\begin{proof}
Follows from the theorem \ref{th1} and theorem \ref{th3}.
\end{proof}
\section{The decomposition theorem}\label{sec}
In this section we prove a decomposition theorem for $I$-convergent sequence in cone normed spaces.
\begin{thm}\label{thDe}
Let $(X, \norm{.}_{\text{P}})$ be a cone normed space, P be a normal cone and $(X,d)$ be a cone metric space where $d$ is the cone metric induced by the cone norm $\norm{.}_\text{P}$. Also let $x=(x_n)$ be a sequence in $X$ and $I\subset 2^\mathbb{N}$ be an admissible ideal with the property (AP). Then the following condition are equivalent:\\
(a) $I\text-\lim x_n=\xi $\\
(b) There exists sequences $y=(y_n)\in X$ and $z=(z_n)$ in $X$ such that $x= y + z $, $d(y_n,\xi) \to \theta_E$ as $n\to\infty$ and $\text{supp z}\in I$, where $\text{supp z}=\left\{n\in \mathbb{N} : z_n\neq \theta_X \right\}$, where $\theta_X$ is the zero element of $X$.
\end{thm}
\begin{proof}
Let $I\text-\displaystyle{\lim_{n\to \infty}}x_n=\xi$. Since $I$ has the property (AP), we conclude that there exists a set $M\in F(I)$, $M=\left\{m_1<m_2<\cdots<m_k <\cdots\right\}$ such that $\displaystyle{\lim_{k \to\infty}}x_{m_k}=\xi$. Now we define a sequence $y=(y_n)$ in X as follows:
 \begin{equation}\label{eq3}
y_n=
\begin{cases}x_n       \hspace{2cm}  \: n\in M.\\
\xi                     \hspace{2cm}  \:n\in \mathbb{N}\setminus M.
\end{cases}
\end{equation}
Thus we can see that $\displaystyle{\lim_{n\to \infty}}y_n=\xi$. Since P is normal so $d(y_n,\xi)\to \theta_E$ as $n\to \infty$. Further put $z_n= x_n - y_n$, $n\in\mathbb{N}$. Since $\left\{k\in \mathbb{N} : x_k\neq y_k\right\} \subset \mathbb{N}\setminus M \in I$. We have $\left\{k\in \mathbb{N} : z_k\neq \theta_X \right\} \in I$ . So it follows that $\text{supp z}\in I$. Also $x=y + z$.\\

Conversely, Suppose there exists two sequence $y=(y_n)\in X $ and $z=(z_n)\in X$ such that $x= y + z$, $d(y_n, \xi )\to \theta_E$ as $n\to \infty$ and $\text{supp z}\in I$. We will prove that $I\text-\displaystyle{\lim_{n \to \infty}} x_n=\xi$. Define a subset $M$ of $\mathbb{N}$ by $M=\left\{m\in \mathbb{N}: z_m\neq \theta_X\right\}$. Now since $\text{supp z}=\left\{m\in \mathbb{N}: z_m\neq \theta_X\right\}\in I$, we have $M\in F(I)$; hence $x_n=y_n$ if $n\in M$. Therefore there exists a set $M=\left\{m_1<m_2<\cdots\right\}\text{,}\: M\in F(I)$ such that $\displaystyle{\lim_{k\to\infty}}x_{m_k}=\xi$. Now as $I$ satisfies the condition (AP), therefore we can say that $I\text-\displaystyle{\lim_{n\to\infty}}x_n=\xi$. Hence the proof is complete.
\end{proof}
\begin{lem}\label{lem1}
Let $(X,\norm{.}_{\text{P}})$ be a cone normed space, P be a normal cone and $(X,d)$ be a cone metric space where $d$ is the cone metric induced by the cone norm $\norm{.}_\text{P}$. Then if $I\text-\lim x_n=\xi$, $I\text-\lim y_n=\eta$ then $I\text-\lim (x_n + y_n)=\xi + \eta$
\end{lem}
\begin{proof}
Since $I\text-\lim x_n=\xi$ and $I\text-\lim y_n=\eta$. So for any $c\in E$ with $\theta_E<< c$ \\
$\left\{n\in \mathbb{N}: \frac{c}{2} - d(x_n,\xi) \not\in \:\text{Int P}\right\}\in I\:\text{and}\: \left\{n\in \mathbb{N}: \frac{c}{2} - d(y_n,\eta) \not \in \:\text{Int P} \right\}\in I$.\\
Now we claim that 
\begin{equation}\label{eq2}
\begin{split}
\left\{n\in \mathbb{N} : c - d(x_n + y_ n, \xi + \eta) \not \in \:\text{Int P}\right\}\subset \left\{n\in \mathbb{N}: \frac{c}{2} - d(x_n,\xi) \not\in \:\text{Int P}\right\} \\
\cup \left\{n\in \mathbb{N}: \frac{c}{2} - d(y_n,\eta) \not \in \:\text{Int P} \right\}.
\end{split}
\end{equation} For, let $m\not \in \left\{n\in \mathbb{N}: \frac{c}{2} - d(x_n,\xi) \not\in \:\text{Int P}\right\} \cup \left\{n\in \mathbb{N}: \frac{c}{2} - d(y_n,\eta) \not \in \:\text{Int P} \right\}$. Then $ \frac{c}{2} - d(x_m,\xi) \in \:\text{Int P}$ and $\frac{c}{2} - d(y_m , \eta)\in \:\text{Int P}$. So $ c - ( d(x_m, \xi) + d(y_m, \eta)) \in \:\text{Int P}$ that is $d(x_m,\xi) + d(y_m,\eta) << c$. Now, \begin{equation}\label{eq1}
d(x_m + y_m , \xi + \eta ) \leq d(x_m + y_m , x_m + \eta) + d(x_m + \eta , \xi +\eta)
\end{equation} 
Now for the cone metric $d$ induced by a cone norm on a cone normed space satisfies $d(x + c , y + c) = d(x,y)$. Thus (\ref{eq1}) implies $d(x_m + y_m , \xi +\eta)\leq d(y_m , \eta) + d(x_m , \xi)$. So we get $d(x_m + y_m , \xi + \eta)\leq d(y_m, \eta)+ d(x_m , \xi) << c$. Hence $c - d(x_m + y_m , \xi + \eta) \in \: \text{Int P}$. Thus $m\not \in \left\{n \in \mathbb{N} : c - d(x_n + y_n , \xi + \eta)\not \in \:\text{int P}\right\}$. Therefore our claim is true. Now the right hand sides of (\ref{eq2}) belongs to $I$. So left hand of (\ref{eq2}) also belongs to $I$. So, $I\text-\lim( x_n + y_n)=\xi + \eta $
\end{proof}
\begin{cor}
$I\text-\displaystyle{\lim_{n\to\infty}}x_n=\xi$ if and only if there exists $(y_n)\in X$ and $(z_n)\in X$ such that $x_n= y_n + z_n$, $d(y_n , \xi)\to \theta_E$ as $n\to\infty$ and $I\text-\displaystyle{\lim_{n\to\infty}}z_n=\theta_X$
\end{cor}
\begin{proof}
Let $z_n= x_n - y_n$, where $(y_n)$ is the sequence defined by (\ref{eq3}). Then $d(y_n , \xi)\to \theta_E$ as $n\to \infty$. Also $I\text-\displaystyle{\lim_{n\to\infty}}z_n = \theta_X$; as for any $c\in E$ with $\theta_E<< c$, $\left\{n\in\mathbb{N} : c - d(z_n , \theta_X) \not \in \:\text{Int P}\right\} \subset \text{supp z} \in I$.\\
Further let $x_n= y_n + z_n $, where $d(y_n , \xi )\to \theta_E$ as $n\to \infty$ and $I\text-\displaystyle{\lim_{n\to\infty}}z_n = \theta_X$. Since $I\text-\displaystyle{\lim_{n\to\infty}}y_n =\xi$, then by lemma (\ref{lem1}) we get $\text{I}\text-\displaystyle{\lim_{n\to\infty}}x_n=\xi$
\end{proof}
\begin{rem}\label{rem2}
From the proof of the theorem (\ref{thDe}), it is clear that (b) implies (a) even if the ideal $I$ have not the property (AP), as $I^*$-convergence implies $I$-convergence. In fact, let $x_n = y_n + z_n $ ,$ d(y_n, \xi) \to \theta_E$ as $n\to\infty$ and $\text{supp z}\in I$, where $I$ is an admissible ideal which has not the property (AP). Since $\left\{n\in\mathbb{N}: c - d(z_n, \theta_X) \not \in \text{Int P}\right\} \subset \left\{n\in \mathbb{N}: z_n\neq \theta_X\right\}\in I$ for any $c\in E$ with $\theta_E<< c$; we have $I\text-\displaystyle{\lim_{n\to\infty}}z_n= \theta_X$. Hence $I\text-\displaystyle{\lim_{n\to\infty}}x_n=\xi$
\end{rem}

\end{document}